\documentclass[12pt]{article}%
\usepackage[intlimits]{amsmath}
\usepackage{amssymb}
\usepackage{latexsym}
\usepackage{tikz}
\usepackage{anysize}
\usepackage{subfigure}
\usepackage[T1]{fontenc}
\usepackage[sc]{mathpazo}
\usepackage{color}
\usepackage[colorlinks]{hyperref}
\usepackage{amsfonts}
\usepackage{graphicx}%
\definecolor {refcol}{RGB}{40,0,255}
\hypersetup{colorlinks=true,allcolors=refcol}
\makeatletter
\newfont{\footsc}{cmcsc10 at 8truept}
\newfont{\footbf}{cmbx10 at 8truept}
\newfont{\footrm}{cmr10 at 10truept}
\newtheorem{theorem}{Theorem}

\newtheorem{corollary}[theorem]{Corollary}

\newtheorem{definition}[theorem]{Definition}
\newtheorem{example}[theorem]{Example}

\newtheorem{lemma}[theorem]{Lemma}

\newenvironment{proof}[1][Proof]{\noindent{\textbf {#1}  }}  {\hfill$\Box$\bigskip}
\begin{document}

\title{\textbf{$A_{\alpha}$-spectrum of a graph obtained by copies of a rooted graph
and applications}}
\author{Oscar Rojo\thanks{Department of Mathematics, Universidad Cat\'{o}lica del
Norte, Antofagasta, Chile.}}
\date{}
\maketitle

\begin{abstract}
Given a connected graph $R$ on $r$ vertices and a rooted graph $H,$ let
$R\{H\}$ be the graph obtained from $r$ copies of $H$ and the graph $R$ by
identifying the root of the $i-th$ copy of $H$ with the $i-th$ vertex of $R$.
Let $0\leq\alpha\leq1,$ and let
\[
A_{\alpha}(G)=\alpha D(G)+(1-\alpha)A(G)
\]
where $D(G)$ and $A(G)$ are the diagonal matrix of the vertex degrees of $G$
and the adjacency matrix of $G$, respectively. A basic result on the
$A_{\alpha}-$ spectrum of $R\{H\}$ is obtained. This result is used to prove
that if $H=B_{k}$ is a generalized Bethe tree on $k$ levels, then the
eigenvalues of $A_{\alpha}(R\{B_{k}\})$ are the eigenvalues of symmetric
tridiagonal matrices of order not exceeding $k$; additionally, the
multiplicity of each eigenvalue is determined. Finally, applications to a
unicyclic graph are given, including an upper bound on the $\alpha-$ spectral
radius in terms of the largest vertex degree and the largest height of the
trees obtained by removing the edges of the unique cycle in the graph.

\end{abstract}



\section{Introduction}

Let $G=\left(  V(G),E(G) \right)  $ be a simple undirected graph on $n$
vertices with vertex set $V(G)$ and edge set $E(G)$. Let $D(G)$ be the
diagonal matrix of order $n$ whose $(i,i)-$entry is the degree of the $i-th$
vertex of $G$ and let $A\left(  G\right) $ be the adjacency matrix of $G.$ The
matrices $L(G)=D(G)-A(G)$ and $Q(G) =D(G)+A(G)$ are the Laplacian and signless
Laplacian matrix of $G$, respectively. The matrices $L(G)$ and $Q(G)$ are both
positive semidefinite and $\left(  0,\mathbf{1}\right) $ is an eigenpair of
$L\left(  G\right) $ where $\mathbf{1}$ is the all ones vector. For a
connected graph $G$, the smallest eigenvalue of $Q(G)$ is positive if and only
if $G$ is non-bipartite.

In \cite{Nik16}, the family of matrices $A_{\alpha}(G)$,
\[
A_{\alpha}(G)=\alpha D(G)+(1-\alpha)A(G)
\]
with $\alpha\in[0,1]$, is introduced together with a number of some basic
results and several open problems.

Observe that $A_{0}\left(  G\right)  =A\left(  G\right)  $ and $A_{1/2}\left(
G\right)  =\frac{1}{2} Q\left(  G\right)  $.

Let $R$ be a connected graph on $r$ vertices. Let $v_{1},\ldots,v_{r}$ be the
vertices of $R$. Let $\varepsilon_{i,j}=\varepsilon_{j,i}=1$ if $v_{i}\sim
v_{j}$ and let $\varepsilon_{i,j}=\varepsilon_{j,i}=0$ otherwise.

We recall that a rooted graph is a graph in which one vertex has been
distinguished as the root and that the level of a vertex is one more than its
distance from the root. In particular, a generalized Bethe tree is a rooted
tree in which vertices at the same level have the same degree.

Let $H$ be a rooted graph. Let $R\{H\}$ be the graph obtained from $R$ and $r$
copies of $H$ by identifying the root of the $i-$copy of $H$ with the vertex
$v_{i}$ of $R$.

In this paper, we obtain a general result on the $A_{\alpha}-$ spectrum of
$R\{H\}$. We use this result to prove that if $H=B_{k}$ is a generalized Bethe
tree on $k$ levels, then the eigenvalues of $A_{\alpha}(R\{B_{k}\})$ are the
eigenvalues of symmetric tridiagonal matrices of order not exceeding $k$;
additionally, the multiplicity of each eigenvalue is determined. Finally, we
apply these results to a unicyclic graph, including the derivation of an upper
bound on the $\alpha-$ spectral radius in terms of the largest vertex degree
and the largest height of the trees obtained by removing the edges of the
unique cycle in the graph.

%Let $\rho(M)$ be the spectral radius of a matrix $M$ and let $\Delta(G)$ be the largest vertex degree of a graph $G$.

\section{A basic result on the $A_{\alpha}-$ spectrum of copies of a rooted
graph}

Let $E$ be the matrix of order $n\times n$ with $1$ in the $\left(
n,n\right)  -$entry and zeros elsewhere. For $i=1,2,...,r,$ let $d\left(
v_{i}\right)  $ be the degree of $v_{i}$ as a vertex of $R$ and let $n$ be the
order of $H$. Then the total number of vertices in $R\{H\}$ is $rn.$ We label
the vertices of $R\{H\}$ as follows: for $i=1,2,\ldots,r,$ using the labels
$\left(  i-1\right)  n+1,\left(  i-1\right)  n+2,\ldots\ ,in,$ we label the
vertices of the $i-th$ copy of $H$ from the vertices at the bottom (the set of
vertices at the largest distance from the root) to the vertex $v_{i}$.

From now on, let $\alpha\in\lbrack0,1]$ and let $\beta=1-\alpha$. With the
above labeling, we obtain $A_{\alpha}(R\{H\})$=%
\begin{equation}
\left[
\begin{array}
[c]{ccccc}%
A_{\alpha}(H)+\alpha d(v_{1})E & \beta\varepsilon_{1,2}E & \cdots & \cdots &
\beta\varepsilon_{1,r}E\\
\beta\varepsilon_{1,2}E & \ddots & \ddots &  & \beta\varepsilon_{2,r}E\\
\vdots & \ddots & \ddots & \ddots & \vdots\\
\vdots &  & \ddots & A_{\alpha}(H)+\alpha d(v_{r-1})E & \beta\varepsilon
_{r-1,r}E\\
\beta\varepsilon_{1,r}E & \beta\varepsilon_{2,r}E & \cdots & \beta
\varepsilon_{r-1,r}E & A_{\alpha}(H)+\alpha d(v_{r})E
\end{array}
\right]  .\label{Mx}%
\end{equation}

In this paper, the identity matrix of appropriate order is denoted by $I$ and
$I_{m}$ denotes the identity matrix of order $m$. Furthermore, we need the
following additional notation: $\left\vert M\right\vert $ and $\phi
_{M}(\lambda)$ denote the determinant and the characteristic polynomial of
$M$, respectively, and $B^{T}$ denotes the transpose of $B$.

The Kronecker product \cite{Zhang} of two matrices $A=\left(  a_{i,j}\right)
$ and $B=\left(  b_{i,j}\right)  $ of sizes $m\times m$ and $n\times n,$
respectively, is the $\left(  mn\right)  \times\left(  mn\right)  $ matrix
$A\otimes B=\left(  a_{i,j}B\right)  .$ Then, in particular, $I_{n}\otimes
I_{m}=I_{nm}$. Some basic properties of the Kronecker product are $\left(
A\otimes B\right)  ^{T}=A^{T}\otimes B^{T}$ and $\left(  A\otimes B\right)
\left(  C\otimes D\right)  =AC\otimes BD$ for matrices of appropriate sizes.
Moreover, if $A$ and $B$ are invertible matrices then $\left(  A\otimes
B\right)  ^{-1}=A^{-1}\otimes B^{-1}.$

Let $\mathrm{Spec}(M)$ be the spectrum of a matrix $M$.

\begin{theorem}
\label{general}Let $\rho_{1}(\alpha),\rho_{2}(\alpha),\ldots,\rho_{r}(\alpha)$
be the eigenvalues of $A_{\alpha}(R)$. Then%
\begin{equation}
\mathrm{Spec}(A_{\alpha}(R\{H\})) = \cup_{j=1}^{r}\mathrm{Spec}(A_{\alpha}(H)
+\rho_{j}(\alpha) E).\label{rH}%
\end{equation}

\end{theorem}

\begin{proof}
From $\left(  \ref{Mx}\right)  $%
\[
A_{\alpha}(R\{H\})=I_{r}\otimes A_{\alpha}(H)+A_{\alpha}(R)\otimes E.
\]
Let
\[
V=\left[
\begin{array}
[c]{ccccc}%
\mathbf{v}_{1} & \mathbf{v}_{2} & \ldots & \mathbf{v}_{r-1} & \mathbf{v}_{r}%
\end{array}
\right]
\]
be an orthogonal matrix whose columns $\mathbf{v}_{1},\mathbf{v}_{2}%
,\ldots,\mathbf{v}_{r}$ are eigenvectors corresponding to the eigenvalues
$\rho_{1}(\alpha),\rho_{2}(\alpha),\ldots,\rho_{r}(\alpha)$, respectively.
Then%
\begin{align*}
(V\otimes I_{n})A_{\alpha}(R\{H\})(V^{T}\otimes I_{n})  & =(V\otimes
I_{n})(I_{r}\otimes A_{\alpha}(H)+A_{\alpha}(R)\otimes E)(V^{T}\otimes
I_{n})\\
& =I_{r}\otimes A_{\alpha}(H)+(VA_{\alpha}(R)V^{T})\otimes E.
\end{align*}
Moreover,
\begin{align*}
(VA_{\alpha}(R)V^{T})\otimes E  & =\\
& \left[
\begin{array}
[c]{ccccc}%
\rho_{1}(\alpha) &  &  &  & \\
& \rho_{2}(\alpha) &  &  & \\
&  & \ddots &  & \\
&  &  & \ddots & \\
&  &  &  & \rho_{r}(\alpha)
\end{array}
\right]  \otimes E\\
& =\\
& \left[
\begin{array}
[c]{ccccc}%
\rho_{1}(\alpha)E &  &  &  & \\
& \rho_{2}(\alpha)E &  &  & \\
&  & \ddots &  & \\
&  &  & \ddots & \\
&  &  &  & \rho_{r}(\alpha)E
\end{array}
\right]  .
\end{align*}
Therefore,%
\begin{align*}
(V\otimes I_{n})A_{\alpha}(R\{H\})(V^{T}\otimes I_{n})  & =\\
& \left[
\begin{array}
[c]{ccccc}%
A_{\alpha}(H)+\rho_{1}(\alpha)E &  &  &  & \\
& A_{\alpha}(H)+\rho_{2}(\alpha)E &  &  & \\
&  & \ddots &  & \\
&  &  & \ddots & \\
&  &  &  & A_{\alpha}(H)+\rho_{r}(\alpha)E
\end{array}
\right]  .
\end{align*}
Since $A_{\alpha}(R\{H\})$ and $(V\otimes I_{n})A_{\alpha}(R\{H\})(V^{T}%
\otimes I_{n})$ are similar matrices, $\left(  \ref{rH}\right)  $ follows.
\end{proof}

\section{$A_{\alpha}$-spectrum of copies of a generalized Bethe tree}

From now on, let $B_{k}$ be a generalized Bethe tree of $k$ levels. From
Theorem \ref{general}, we have

\begin{theorem}
\label{RBk}Let $\rho_{1}(\alpha),\rho_{2}(\alpha),\ldots,\rho_{r}(\alpha)$ be
the eigenvalues of $A_{\alpha}(R)$. Then%
\[
\mathrm{Spec}(A_{\alpha}(R\{B_{k}\})) = \cup_{i=1}^{r}\mathrm{Spec}(A_{\alpha
}(B_{k}) +\rho_{i}(\alpha) E).
\]

\end{theorem}

For $1\leq j\leq k$, let $n_{j}$ and $d_{j}$ be the number and the degree of
the vertices of $B_{k}$ at the level $k-j+1,$ respectively. Thus $d_{k}$ is
the degree of the root, $n_{k}=1,$ $d_{1}=1$ and $n_{1}$ is the number of
pendant vertices. For $1 \leq j \leq k-1$, let $m_{j}=\frac{n_{j}}{n_{j+1}}$.

\begin{definition}
\label{defP}Let%
\[
P_{0}\left(  \lambda\right)  =1,\;P_{1}\left(  \lambda\right)  =\lambda
-\alpha,
\]
and
\[
P_{j}(\lambda)=(\lambda-\alpha d_{j})P_{j-1}(\lambda)-\beta^{2}m_{j-1}%
P_{j-2}(\lambda)
\]
for $j=2,\ldots,k$.
\end{definition}

For brevity, sometimes we write $f$ instead $f(\lambda)$. The polynomials in
Definition \ref{defP} are used in \cite{NPRS16}, Theorem $5$, to factor the
characteristic polynomial of $A_{\alpha}(B_{k})$ as given below.

\begin{theorem}
\label{poly} The characteristic polynomial of $A_{\alpha}(B_{k})$ satisfies%
\begin{equation}
\phi_{A_{\alpha}(B_{k})}(\lambda) =P_{1}^{n_{1}-n_{2}}P_{2}^{n_{2}-n_{3}%
}\ldots P_{k-2}^{n_{k-2}-n_{k-1}}P_{k-1}^{n_{k-1}-1}P_{k}. \label{pA}%
\end{equation}

\end{theorem}

At this point, we introduce the notation $\widetilde{M}$ to mean the matrix
obtained from $M$ by deleting its last row and its last column.

For $1 \leq i \leq r$, consider the matrices $A_{\alpha}(B_{k}) +\rho
_{i}(\alpha) E$ in Theorem \ref{RBk}. Let
\[
M_{i}(\alpha) = A_{\alpha}(B_{k}) +\rho_{i}(\alpha) E.
\]

Then
\[
\phi_{M_{i}(\alpha)}(\lambda) = |\lambda I - A_{\alpha}(B_{k}) - \rho
_{i}(\alpha) E|.
\]

Applying linearity on the last column, we obtain
\begin{equation}
\label{t5}\phi_{M_{i}(\alpha)}(\lambda)= |\lambda I - A_{\alpha}(B_{k})| -
\rho_{i}(\alpha)|\lambda I - \widetilde{A_{\alpha}(B_{k})}|=\\
\phi_{A_{\alpha}(B_{k})}(\lambda)-\rho_{i}(\alpha)\phi_{\widetilde{A_{\alpha
}(B_{k})}}(\lambda).
\end{equation}

Theorem \ref{poly} gives $\phi_{A_{\alpha}(B_{k})}(\lambda)$ as a product of
powers of the polynomials $P_{j}(\lambda)$ ($1 \leq j \leq k$).

We now focus our attention on $\phi_{\widetilde{A_{\alpha}(B_{k})}}(\lambda)$.
From the proof of Theorem 5 in \cite{NPRS16}, we have
\[
\phi_{A_{\alpha}(B_{k})}(\lambda)=\phi_{\widetilde{A_{\alpha}(B_{k})}}%
(\lambda)\frac{P_{k}}{P_{k-1}}.
\]
From this identity and Theorem \ref{poly}, we obtain
\[
\phi_{\widetilde{A_{\alpha}(B_{k})}}(\lambda)=P_{1}^{n_{1}-n_{2}}P_{2}%
^{n_{2}-n_{3}}\ldots P_{k-2}^{n_{k-2}-n_{k-1}}P_{k-1}^{n_{k}}%
\]
Replacing in (\ref{t5}) and factoring, we get

\begin{lemma}
\label{polyRBk} Let $\rho_{1}(\alpha),\rho_{2}(\alpha),\ldots,\rho_{r}%
(\alpha)$ be the eigenvalues of $A_{\alpha}(R)$. For $i=1,\ldots,r$, the
characteristic polynomial of $M_{i}(\alpha)=A_{\alpha}(B_{k}) +\rho_{i}%
(\alpha) E$ satisfies
\[
\phi_{M_{i}(\alpha)}(\lambda)=P_{1}^{n_{1}-n_{2}}P_{2}^{n_{2}-n_{3}}\ldots
P_{k-2}^{n_{k-2}-n_{k-1}}P_{k-1}^{n_{k}-1}(P_{k}-\rho_{i}(\alpha) P_{k-1}).
\]

\end{lemma}

\begin{definition}
\label{T} For $j=1,2,\ldots,k-1$, let $T_{j}$ be the $j\times j$ leading
principal submatrix of the $k\times k$ symmetric tridiagonal matrix
\begin{equation}
T_{k}=\left[
\begin{array}
[c]{ccccc}%
\alpha & \beta\sqrt{d_{2}-1} & 0 &  & 0\\
\beta\sqrt{d_{2}-1} & \alpha d_{2} & \ddots &  & \\
& \ddots & \ddots & \beta\sqrt{d_{k-1}-1} & \\
&  & \beta\sqrt{d_{k-1}-1} & \alpha d_{k-1} & \beta\sqrt{d_{k}}\\
0 &  & 0 & \beta\sqrt{d_{k}} & \alpha d_{k}%
\end{array}
\right] . , \label{mat}%
\end{equation}

\end{definition}

Since $d_{s}>1$ for all $s=2,3,....,j$, each matrix $T_{j}$ has nonzero
codiagonal entries and it is known that its eigenvalues are simple.

The relationship between these matrices and the polynomials $P_{j}(\lambda)$
is given in \cite{NPRS16}, Lemma $7$:

\begin{lemma}
\label{TP} For $j=1,\ldots,k$,
\[
\phi_{T_{j}}(\lambda) = P_{j}(\lambda).
\]

\end{lemma}

\begin{definition}
\label{M} For $i=1,\ldots,r$, let $Q_{i}(\lambda)$ be the polynomial
\[
Q_{i}(\lambda)=P_{k}(\lambda)-\rho_{i}(\alpha) P_{k-1}(\lambda)
\]
and $S_{i}$ be the $k \times k$ matrix
\[
S_{i}=\left[
\begin{array}
[c]{ccccc}%
\alpha & \beta\sqrt{d_{2}-1} & 0 &  & 0\\
\beta\sqrt{d_{2}-1} & \alpha d_{2} & \ddots &  & \\
& \ddots & \ddots & \beta\sqrt{d_{k-1}-1} & \\
&  & \beta\sqrt{d_{k-1}-1} & \alpha d_{k-1} & \beta\sqrt{d_{k}}\\
0 &  & 0 & \beta\sqrt{d_{k}} & \alpha d_{k} + \rho_{i}(\alpha)
\end{array}
\right] .
\]

\end{definition}

The relationship between the polynomials $Q_{i}$ and the matrices $S_{i}$ is
given in the following lemma.

\begin{lemma}
\label{MR} For $i=1,\ldots,r$, $Q_{i}(\lambda)$ is the characteristic
polynomial of the $k \times k$ matrix $S_{i}$, that is,
\[
\phi_{S_{i}}(\lambda)=Q_{i}(\lambda).
\]

\end{lemma}

\begin{proof}
Applying linearity on the last column, we obtain
\[
\phi_{S_{i}}(\lambda)=|\lambda I-S_{i}|=|\lambda I-T_{k}|-\rho_{i}%
(\alpha)|\lambda I-T_{k-1}|.
\]
Now Lemma \ref{TP} implies that
\[
\phi_{S_{i}}(\lambda)=P_{k}(\lambda)-\rho_{i}(\alpha)P_{k-1}(\lambda
)=Q_{i}(\lambda).
\]

\end{proof}

We are ready to state the main result of this section.

\begin{theorem}
\label{specRBk}Let $B_{k}$ be a generalized Bethe tree on $k$ levels, and
$\alpha\in\left[  0,1\right)  $. Let $\rho_{1}(\alpha)\ldots\rho_{r}(\alpha)$
be the eigenvalues of $A_{\alpha}(R)$ in which $\rho_{1}(\alpha)$ is the
spectral radius. If the matrices $T_{1},\ldots,T_{k}$ and $S_{1},\ldots,S_{r}$
are as in Definitions \ref{T} and \ref{M}, respectively, then

(1)
\begin{equation}
\mathrm{Spec}(A_{\alpha}(R\{B_{k}\}))=(\cup_{j=1}^{k-1}\mathrm{Spec}%
(T_{j}))\cup(\cup_{i=1}^{r}\mathrm{Spec}(S_{i})).
\end{equation}

(2) For $1 \leq j \leq k-1$, the multiplicity of each eigenvalue of $T_{j}$ as
an eigenvalue of $A_{\alpha}(R\{B_{k}\})$ is $r(n_{j}-n_{j+1})$, and for $1
\leq i \leq r$, the eigenvalues of $S_{i}$ as eigenvalues of $A_{\alpha
}(R\{B_{k}\})$ are simple. If some eigenvalues obtained in different matrices
are equal, their multiplicities are added together.

(3) The largest eigenvalue of $S_{1}$ is the spectral radius of $A_{\alpha
}(R\{B_{k}\})$.
\end{theorem}

\begin{proof}
(1) and (2) are consequences of Theorem \ref{RBk}, Lemma \ref{polyRBk}, Lemma
\ref{TP} and Lemma \ref{MR}. The eigenvalues of each $T_{j}$ interlace the
eigenvalues of any $S_{i}$. Then the spectral radius of $A_{\alpha}%
(R\{B_{k}\})$ is the largest of the spectral radii of the matrices $S_{i}$. We
use the fact that the spectral radius of an irreducible nonnegative matrix
increases when any of its entries increases, to obtain item (3).
\end{proof}

A \emph{Bethe tree} $B(d,k)$ is a rooted tree of $k$ levels in which the root
has degree $d$, the vertices at level $j\ \left(  2\leq j\leq k-1\right)  $
have degree $d+1$ and the vertices at level $k$ have degree equal to $1$
(pendant vertices). Clearly, any Bethe tree is a generalized Bethe tree.
Theorem \ref{specRBk} immediately implies the following corollary.

\begin{corollary}
\label{bethe}Let $\alpha\in\left[  0,1\right)  $, and $\beta=1-\alpha$. Let
$\rho_{1}(\alpha) \ldots\rho_{r}(\alpha)$ be the eigenvalues of $A_{\alpha
}(R)$ in which $\rho_{1}(\alpha)$ is the spectral radius. For $1 \leq j \leq
k$, let $T_{j}$ be the leading principal submatrix of order $j\times j$ of the
$k\times k$ symmetric tridiagonal matrix
\[
T_{k}=\left[
\begin{array}
[c]{ccccc}%
\alpha & \beta\sqrt{d} & 0 &  & 0\\
\beta\sqrt{d} & \alpha\left(  d+1\right)  & \beta\sqrt{d} &  & \\
& \ddots & \ddots & \ddots & \\
&  &  & \alpha\left(  d+1\right)  & \beta\sqrt{d}\\
0 &  & 0 & \beta\sqrt{d} & \alpha d
\end{array}
\right]  .
\]
For $1 \leq i \leq r$, let
\[
S_{i}=\left[
\begin{array}
[c]{ccccc}%
\alpha & \beta\sqrt{d} & 0 &  & 0\\
\beta\sqrt{d} & \alpha\left(  d+1\right)  & \beta\sqrt{d} &  & \\
& \ddots & \ddots & \ddots & \\
&  &  & \alpha\left(  d+1\right)  & \beta\sqrt{d}\\
0 &  & 0 & \beta\sqrt{d} & \alpha d + \rho_{i}(\alpha]
\end{array}
\right]  .
\]

Then

(1)
\[
\mathrm{Spec}(A_{\alpha}(R\{B(d,k)\}))=(\cup_{j=1}^{k-1}\mathrm{Spec}%
(T_{j}))\cup(\cup_{i=1}^{r}\mathrm{Spec}(S_{i})).
\]

(2) For $1 \leq j \leq k-1$, the multiplicity of each eigenvalue of $T_{j}$ as
an eigenvalue of $A_{\alpha}(R\{B(d,k)\})$ is $rd^{k-j-1}(d-1)$, and for $1
\leq i \leq r$, the eigenvalues of $S_{i}$ as eigenvalues of $A_{\alpha
}(R\{B(d,k)\})$ are simple. If some eigenvalues obtained in different matrices
are equal, their multiplicities are added together.

(3) The largest eigenvalue of $S_{1}$ is the spectral radius of $A_{\alpha
}(R\{B(d,k)\})$.
\end{corollary}

\section{Applications to unicyclic graphs. An upper bound on the $A_{\alpha}%
$-spectral radius}

In this section we consider $R=C_{r}$, the cycle on $r$ vertices. It is know
that the eigenvalues of the adjacency matrix of $C_{r}$ are $2\cos(\frac
{2\pi(i-1)}{r})$, $1\leq i\leq r$. Since the cycle $C_{r}$ is a $2-$ regular
graph, it follows that the eigenvalues of $A_{\alpha}(C_{r})$ are
\[
\rho_{i}(\alpha)=2\alpha+2(1-\alpha)\cos(\frac{2\pi(i-1)}{r})
\]
for $i=1,\ldots,r$. Hence the spectral radius of $A_{\alpha}(C_{r})$ is
$\rho_{1}(\alpha)=2$ for any $\alpha\in\lbrack0,1]$.

From Theorem \ref{specRBk}, we have

\begin{corollary}
\label{cycleBk}Let $B_{k}$ be a generalized Bethe tree of $k$ levels. Let
$T_{1},\ldots,T_{k}$ be as in Definitions \ref{T}. For $i=1,\ldots,r$, let
$S_{i}$ be the $k \times k$ matrix
\[
S_{i}=\left[
\begin{array}
[c]{ccccc}%
\alpha & \beta\sqrt{d_{2}-1} & 0 &  & 0\\
\beta\sqrt{d_{2}-1} & \alpha d_{2} & \ddots &  & \\
& \ddots & \ddots & \beta\sqrt{d_{k-1}-1} & \\
&  & \beta\sqrt{d_{k-1}-1} & \alpha d_{k-1} & \beta\sqrt{d_{k}}\\
0 &  & 0 & \beta\sqrt{d_{k}} & \alpha d_{k} + 2 \alpha+ 2(1-\alpha)\cos
(\frac{2 \pi(i-1)}{r})
\end{array}
\right] .
\]
Then

(1)
\begin{equation}
\mathrm{Spec}(A_{\alpha}(C_{r}\{B_{k}\}))=(\cup_{j=1}^{k-1}\mathrm{Spec}%
(T_{j}))\cup(\cup_{i=1}^{r}\mathrm{Spec}(S_{i})).
\end{equation}

(2) For $1 \leq j \leq k-1$, the multiplicity of each eigenvalue of $T_{j}$ as
an eigenvalue of $A_{\alpha}(C_{r}\{B_{k}\})$ is $r(n_{j}-n_{j+1})$, and for
$1 \leq i \leq r$, the eigenvalues of $S_{i}$ as eigenvalues of $A_{\alpha
}(C_{r}\{B_{k}\})$ are simple. If some eigenvalues obtained in different
matrices are equal, their multiplicities are added together.

(3) The largest eigenvalue of
\[
S_{1}=\left[
\begin{array}
[c]{ccccc}%
\alpha & \beta\sqrt{d_{2}-1} & 0 &  & 0\\
\beta\sqrt{d_{2}-1} & \alpha d_{2} & \ddots &  & \\
& \ddots & \ddots & \beta\sqrt{d_{k-1}-1} & \\
&  & \beta\sqrt{d_{k-1}-1} & \alpha d_{k-1} & \beta\sqrt{d_{k}}\\
0 &  & 0 & \beta\sqrt{d_{k}} & \alpha d_{k} + 2
\end{array}
\right] .
\]
is the spectral radius of $A_{\alpha}(C_{r}\{B_{k}\})$.
\end{corollary}

For the case of copies of the Bethe tree $B(d,k)$ attached to $C_{r}$, from
Corollary \ref{bethe}, we have

\begin{corollary}
\label{bethecycle}Let $\alpha\in\left[  0,1\right)  $, and $\beta=1-\alpha$.
For $1 \leq j \leq k$, let $T_{j}$ be as Corollary \ref{bethe}. For $1 \leq i
\leq r$, let
\[
S_{i}=\left[
\begin{array}
[c]{ccccc}%
\alpha & \beta\sqrt{d} & 0 &  & 0\\
\beta\sqrt{d} & \alpha\left(  d+1\right)  & \beta\sqrt{d} &  & \\
& \ddots & \ddots & \ddots & \\
&  &  & \alpha\left(  d+1\right)  & \beta\sqrt{d}\\
0 &  & 0 & \beta\sqrt{d} & \alpha d + 2\alpha+2\beta\cos(\frac{2 \pi(i-1)}{r})
\end{array}
\right]  .
\]
Then

(1) The spectrum of $A_{\alpha}(C_{r}\{B(d,k)\})$ is the multiset union
\[
\mathrm{Spec}(A_{\alpha}(C_{r}\{B(d,k\}))=(\cup_{j=1}^{k-1}\mathrm{Spec}%
(T_{j}))\cup(\cup_{i=1}^{r}\mathrm{Spec}(S_{i})).
\]

(2) For $1 \leq j \leq k-1$, the multiplicity of each eigenvalue of $T_{j}$ as
an eigenvalue of $A_{\alpha}(C_{r}\{B(d,k)\})$ is $rd^{k-j-1}(d-1)$, and for
$1 \leq i \leq r$, the eigenvalues of $S_{i}$ as eigenvalues of $A_{\alpha
}(C_{r}\{B(d,k)\})$ are simple. If some eigenvalues obtained in different
matrices are equal, their multiplicities are added together.

(3) The largest eigenvalue of
\[
S_{i}=\left[
\begin{array}
[c]{ccccc}%
\alpha & \beta\sqrt{d} & 0 &  & 0\\
\beta\sqrt{d} & \alpha\left(  d+1\right)  & \beta\sqrt{d} &  & \\
& \ddots & \ddots & \ddots & \\
&  &  & \alpha\left(  d+1\right)  & \beta\sqrt{d}\\
0 &  & 0 & \beta\sqrt{d} & \alpha d + 2
\end{array}
\right]  .
\]
is the spectral radius of $A_{\alpha}(C_{r}\{B(d,k)\})$.
\end{corollary}

Let $\rho(M)$ be the spectral radius of the matrix $M$. It is known that if
$G$ is a subgraph of $H$ then $\rho(L(G)) \leq\rho(L(H))$, $\rho(Q(G))
\leq\rho(Q(H))$ and $\rho(A(G)) \leq\rho(A(H))$.

In \cite{Ste03} Stevanovi\'{c} proves that for a tree $\mathcal{T}$ with
largest vertex degree $\Delta,$%
\[
\rho(L(T)) <\Delta+2\sqrt{\Delta-1}%
\]
and%
\[
\rho(A(T)) <2\sqrt{\Delta-1}.
\]

In \cite{Hu} Hu proves that if $G$ is a unicyclic graph with largest vertex
degree $\Delta$ then%
\begin{equation}
\label{huL}\rho(L(G)) \leq\Delta+2\sqrt{\Delta-1}%
\end{equation}
with equality if and only if $G$ is the cycle $C_{n}$ whenever $n$ is even,
and%
\begin{equation}
\label{huA}\rho(A(G)) \leq2\sqrt{\Delta-1}%
\end{equation}
with equality if and only if $G$ is the cycle $C_{n}.$

From now on, let $G=(V(G),E(G))$ be a unicyclic graph with largest vertex
degree $\Delta$.

We recall that the height of a rooted tree is the largest distance from its
root to a pendant vertex. The following invariant for a unicyclic graph $G$
was introduced in \cite{Ro08}.

\begin{definition}
\label{dd} Let $G$ be a unicyclic graph. Let $C_{r}$ be the unique cycle in
$G$ and let $v_{1},v_{2},\ldots,v_{r}$ be the vertices of $C_{r}.$ The graph
$G - E(G)$ is a forest of $r$ rooted trees $T_{1}, T_{2}, \ldots,T_{r}$ with
roots $v_{1},\ldots,v_{r},$ respectively. For $i=1,2,\ldots,r$, let $h(T_{i})$
be the height of the tree $T_{i}$. Let%
\[
k(G) =\max\left\{  h(T_{i}):1\leq i\leq r\right\}  +1.
\]
We say that $k(G)$ is the height of the unicyclic $G$.
\end{definition}

\begin{example}
Let $G$ be the unicyclic graph:
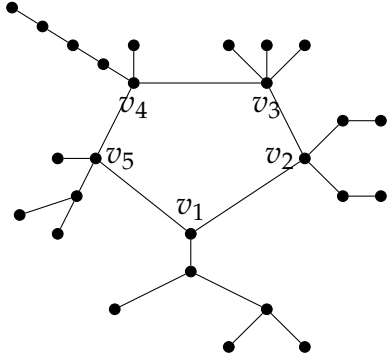
\begin{align*}
\begin{tikzpicture} \tikzstyle{every node}=[draw,circle,fill=black,minimum size=4pt, inner sep=0pt] \draw (-0.25,0) node (1) [label=above:$v_1$] {} (1.25,1) node (2) [label=left:$v_2$] {} (0.75,2) node (3) [label=below:$v_3$] {} (-1,2) node (4) [label=below:$v_4$] {} (-1.5,1) node (5) [label=right:$v_5$] {} (-0.25,-0.5) node (6) [label=right:$$] {} (-1.25,-1) node (7) [label=below:$$] {} (0.75,-1) node (8) [label=right:$$] {} (0.25,-1.5) node (9) [label=below:$$] {} (1.25,-1.5) node (10) [label=below:$$] {} (1.75,0.5) node (11) [label=above:$$] {} (1.75,1.5) node (12) [label=above:$$] {} (2.25,0.5) node (13) [label=above:$$] {} (2.25,1.5) node (14) [label=above:$$] {} (1.25,2.5) node (15) [label=above:$$] {} (0.75,2.5) node (16) [label=above:$$] {} (0.25,2.5) node (17) [label=above:$$] {} (-1,2.5) node (18) [label=above:$$] {} (-1.4,2.25) node (19) [label=below:$$] {} (-1.8,2.5) node (20) [label=below:$$] {} (-2.2,2.75) node (21) [label=below:$$] {} (-2.6,3) node (22) [label=below:$$] {} (-2,1) node (23) [label=below:$$] {} (-1.75,0.5) node (24) [label=right:$$] {} (-2.5,0.25) node (25) [label=below:$$] {} (-2,0) node (26) [label=below:$$] {}; \draw (1)--(2); \draw (2)--(3); \draw (3)--(4); \draw (4)--(5); \draw (1)--(5); \draw (1)--(6); \draw (6)--(7); \draw (6)--(8); \draw (8)--(9); \draw (8)--(10); \draw (2)--(11); \draw (2)--(12); \draw (11)--(13); \draw (12)--(14); \draw (3)--(15); \draw (3)--(16); \draw (3)--(17); \draw (4)--(18); \draw (4)--(19); \draw (19)--(20);\draw (20)--(21); \draw (21)--(22); \draw (5)--(23); \draw (5)--(24); \draw (24)--(25);\draw (24)--(26); \end{tikzpicture}
\end{align*}
Then $\Delta(G)=5$ and the height of $G$ is
\[
k(G)=max\{3,2,1,4,2\}+1=5,
\]

\end{example}

In \cite{Ro08} the upper bounds given in (\ref{huL}) and (\ref{huA}) are
improved as follows:

\begin{lemma}
\label{zero} If $G$ is a unicyclic graph then%
\begin{equation}
\label{rxL}\rho(L(G)) < \Delta+2\sqrt{\Delta-1}\cos\frac{\pi}{2k(G)+1}%
\end{equation}
for $\Delta\geq3$ and
\begin{equation}
\label{rxA}\rho(A(G)) < 2\sqrt{\Delta-1}\cos\frac{\pi}{2k(G)+1}%
\end{equation}
for $\Delta\geq4$ or $\Delta= 3$ and $k(G) \geq4$.
\end{lemma}

It is well known \cite{CvRoSi07} that
\[
\rho(L(G)) \leq\rho(Q(G))
\]
with equality if and only if $G$ is a bipartite graph. In \cite{hansen} upper
bounds on $\rho(L(G))$ and $\rho(Q(G))$ are given and it is proved that many
but not all upper bounds on $\rho(L(G))$ are also bounds for $\rho(Q(G))$. In
\cite{CoPiRo} it is shown that if $G$ is a unicyclic graph, the upper bound on
$\rho(L(G))$ in $(\ref{rxL})$ is also an upper bound on $\rho(Q(G))$.

\begin{lemma}
\label{extra} Let $\Delta\geq3$. Let
\[
X=\left[
\begin{array}
[c]{ccccc}%
0 & \sqrt{\Delta-1} &  &  & \\
\sqrt{\Delta-1} & 0 & \ddots &  & \\
& \ddots & \ddots & \sqrt{\Delta-1} & \\
&  & \sqrt{\Delta-1} & 0 & \sqrt{\Delta-2}\\
&  &  & \sqrt{\Delta-2} & 2\\
&  &  &  &
\end{array}
\right]
\]
be a tridiagonal matrix of order $k\times k$. If $\Delta\geq4$ or $\Delta= 3$
with $k \geq4$ then
\begin{equation}
\label{XY}\rho(X) < 2 \sqrt{\Delta- 1} \cos\frac{\pi}{2 k +1}.
\end{equation}

\end{lemma}

\begin{proof}
Let
\[
Y=\left[
\begin{array}
[c]{ccccc}%
0 & \sqrt{\Delta-1} &  &  & \\
\sqrt{\Delta-1} & 0 & \ddots &  & \\
& \ddots & \ddots & \sqrt{\Delta-1} & \\
&  & \sqrt{\Delta-1} & 0 & \sqrt{\Delta-1}\\
&  &  & \sqrt{\Delta-1} & \sqrt{\Delta-1}\\
&  &  &  &
\end{array}
\right]
\]
be a symmetric tridiagonal matrix of order $k\times k$. It is known \cite{Ko}
that
\[
\rho(Y)=2\sqrt{\Delta-1}\cos\frac{\pi}{2k+1}.
\]
Hence proving (\ref{XY}) is equivalent to proving that $\rho(X)<\rho(Y)$.
Suppose that $\Delta\geq5$. Then $X\leq Y$ with strict inequalities in the
entries $(k-1,k)$ and $(k,k-1)$. Since the spectral radius of an irreducible
nonnegative matrix increases when any of its entries increases, we have
$\rho(X)<\rho(Y)$. Thus, (\ref{XY}) has been proved for $\Delta\geq5$. For
$j=1,2,\ldots,k$, let $x_{j}(\lambda)$ and $y_{j}(\lambda)$ be the
characteristic polynomials of the $j\times j$ leading principal submatrices of
$X$ and $Y$, respectively. Notice that $x_{j}(\lambda)$ and $y_{j}(\lambda)$
are identical polynomials for $j=1,2,\ldots,k-1$. Using the three-term
recursion formula for symmetric tridiagonal matrices, we have
\begin{equation}
x_{k}(\lambda)=(\lambda-2)x_{k-1}(\lambda)-(\Delta-2)x_{k-2}(\lambda
)\label{xk}%
\end{equation}
and
\begin{equation}
y_{k}(\lambda)=(\lambda-\sqrt{\Delta-1})x_{k-1}(\lambda)-(\Delta
-1)x_{k-2}(\lambda).\label{yk}%
\end{equation}
Subtracting $(\ref{yk})$ from $(\ref{xk}),$ we obtain%
\begin{equation}
x_{k}(\lambda)-y_{k}(\lambda)=(\sqrt{\Delta-1}-2)x_{k-1}(\lambda
)+x_{k-2}(\lambda).\label{ddd}%
\end{equation}
Since $X$ and $Y$ are symmetric tridiagonal matrices with nonzero codiagonal
entries, their eigenvalues are simple. Let
\[
\alpha_{k}<\alpha_{k-1}<\ldots<\alpha_{2}<\rho(X)
\]
be the eigenvalues of $X$. Then
\[
x_{k}(\lambda)=(\lambda-\rho(X))\prod_{j=2}^{k}(\lambda-\alpha_{j}).
\]
Let $\beta_{1}$ be the largest zero of the identical polynomials
$x_{k-1}(\lambda)$ and $y_{k-1}(\lambda).$ Since the zeros of these
polynomials strictly interlace the zeros of the polynomials $x_{k}(\lambda)$
and $y_{k}(\lambda),$ we have $\alpha_{2}<\beta_{1}<\rho(X)$ and $\beta
_{1}<\rho(Y).$ Therefore $\alpha_{2}<\rho(Y),$ $x_{k-1}(\rho(Y))>0$ and%
\[
x_{k}(\rho(Y))=(\rho(Y)-\rho(X))c,
\]
where
\[
c=\prod_{j=2}^{k}(\rho(Y)-\alpha_{j})>0.
\]
Thus in order to conclude that $\rho(X)<\rho(Y)$, we need to show that
$x_{k}(\rho(Y))>0$. From $(\ref{yk})$ and $(\ref{ddd}),$%
\[
y_{k}(\rho(Y))=0=(\rho(Y)-\sqrt{\Delta-1})x_{k-1}(\rho(Y))-(\Delta
-1)x_{k-2}(\rho(Y))
\]
and%
\[
x_{k}(\rho(Y))=(\sqrt{\Delta-1}-2)x_{k-1}(\rho(Y))+x_{k-2}(\rho(Y)).
\]
Then%
\begin{equation}
x_{k}(\rho(Y))=(\sqrt{\Delta-1}-2+\frac{\rho(Y)-\sqrt{\Delta-1}}{\Delta
-1})x_{k-1}(\rho(Y)).\label{xb}%
\end{equation}
Let $\Delta=4.$ From $(\ref{xb})$%
\begin{align*}
x_{k}(\rho(Y))  & =(\sqrt{3}-2+\frac{2\sqrt{3}\cos\frac{\pi}{2k+1}-\sqrt{3}%
}{3})x_{k-1}(\rho(Y))\\
& =(\frac{2\sqrt{3}}{3}-2+\frac{2\sqrt{3}}{3}\cos\frac{\pi}{2k+1})x_{k-1}%
(\rho(Y))\\
& \geq(\frac{2\sqrt{3}}{3}-2+\frac{2\sqrt{3}}{3}\cos\frac{\pi}{5})x_{k-1}%
(\rho(Y))>0.08x_{k-1}(\rho(Y))>0.
\end{align*}
It remains to prove (\ref{XY}) for $\Delta=3$ and $k\geq4.$ From $(\ref{xb})$%
\begin{align*}
x_{k}(\rho(Y))  & =(\sqrt{2}-2+\frac{2\sqrt{2}\cos\frac{\pi}{2k+1}-\sqrt{2}%
}{2})x_{k-1}(\rho(Y))\\
& =(\frac{\sqrt{2}}{2}-2+\sqrt{2}\cos\frac{\pi}{2k+1})x_{k-1}(\rho(Y))\\
& \geq(\frac{\sqrt{2}}{2}-2+\sqrt{2}\cos\frac{\pi}{9})x_{k-1}\rho
(Y))>0.03x_{k-1}(\rho(Y))>0.
\end{align*}
The proof of Lemma \ref{extra} is complete.
\end{proof}

At this point, we observe that if $G$ is an induced subgraph of $H$ then
$A_{\alpha}(G) \leq A_{\alpha}(H)$.

The next theorem gives an upper bound on the spectral radius of $A_{\alpha
}(G)$ in terms of the largest degree and height of the unicyclic graph $G$.

\begin{theorem}
Let $G$ be a unicyclic graph with $\Delta\geq3$. Let $\alpha\in[0,1)$. If
$\Delta\geq4$ or $\Delta= 3$ and $k(G) \geq4$ then
\begin{equation}
\label{upper}\rho(A_{\alpha}(G)) < \alpha\Delta+ 2 (1-\alpha) \sqrt{\Delta
-1}\cos\frac{\pi}{2k(G)+1}%
\end{equation}

\end{theorem}

\begin{proof}
Let $C_{r}$ be the unique cycle in $G$. Let $H =B_{k(G)}$ be a generalized
Bethe tree with vertex degrees
\[
d_{1}=1, d_{2}=\Delta,\ldots,d_{k(G)-1}=\Delta, d_{k(G)}=\Delta-2
\]
from the pendant vertices to the root. Then $G$ is an induced subgraph of
$C_{r}\{H\}$. Hence $\rho(A_{\alpha}(G)) \leq\rho(A_{\alpha}(C_{r}\{H\}))$.
Let $\beta= 1-\alpha$. From Corollary \ref{cycleBk}, the spectral radius of
$A_{\alpha}(C_{r}\{H\})$ is the spectral radius of the $k(G) \times k(G)$
matrix
\[
S_{1} = \left[
\begin{array}
[c]{ccccc}%
\alpha & \beta\sqrt{\Delta-1} &  &  & \\
\beta\sqrt{\Delta-1} & \alpha\Delta & \ddots &  & \\
& \ddots & \ddots & \beta\sqrt{\Delta-1} & \\
&  & \beta\sqrt{\Delta-1} & \alpha\Delta & \beta\sqrt{\Delta-2}\\
&  &  & \beta\sqrt{\Delta-2} & \alpha(\Delta-2) +2\\
&  &  &  &
\end{array}
\right]
\]
\[
= \left[
\begin{array}
[c]{ccccc}%
\alpha & \beta\sqrt{\Delta-1} &  &  & \\
\beta\sqrt{\Delta-1} & \alpha\Delta & \ddots &  & \\
& \ddots & \ddots & \beta\sqrt{\Delta-1} & \\
&  & \beta\sqrt{\Delta-1} & \alpha\Delta & \beta\sqrt{\Delta-2}\\
&  &  & \beta\sqrt{\Delta-2} & \alpha\Delta+2\beta\\
&  &  &  &
\end{array}
\right] .
\]
We have
\[
S_{1} \leq\left[
\begin{array}
[c]{ccccc}%
\alpha\Delta & \beta\sqrt{\Delta-1} &  &  & \\
\beta\sqrt{\Delta-1} & \alpha\Delta & \ddots &  & \\
& \ddots & \ddots & \beta\sqrt{\Delta-1} & \\
&  & \beta\sqrt{\Delta-1} & \alpha\Delta & \beta\sqrt{\Delta-2}\\
&  &  & \beta\sqrt{\Delta-2} & \alpha\Delta+2\beta\\
&  &  &  &
\end{array}
\right]
\]
\[
=\alpha\left[
\begin{array}
[c]{ccccc}%
\Delta &  &  &  & \\
& \Delta &  &  & \\
&  & \ddots &  & \\
&  &  & \ddots & \\
&  &  & \  & \Delta\\
&  &  &  &
\end{array}
\right]  + \beta\left[
\begin{array}
[c]{ccccc}%
0 & \sqrt{\Delta-1} &  &  & \\
\sqrt{\Delta-1} & \ddots & \ddots &  & \\
& \ddots & \ddots & \sqrt{\Delta-1} & \\
&  & \sqrt{\Delta-1} & 0 & \sqrt{\Delta-2}\\
&  &  & \sqrt{\Delta-2} & 2\\
&  &  &  &
\end{array}
\right]
\]
\[
=\alpha\left[
\begin{array}
[c]{ccccc}%
\Delta &  &  &  & \\
& \Delta &  &  & \\
&  & \ddots &  & \\
&  &  & \ddots & \\
&  &  & \  & \Delta\\
&  &  &  &
\end{array}
\right]  + \beta X
\]
where $X$ is as in Lemma \ref{extra}. Then
\[
\rho(S_{1}) \leq\alpha\Delta+\beta\rho(X).
\]
If $\Delta\geq4$ and $\Delta= 3$ with $k(G) \geq4$, applying Lemma
\ref{extra}, the upper bound (\ref{upper}) follows.
\end{proof}

Finally we study the cases $\Delta=3$ and $k(G)\leq3$. In \cite{Ro08} it is
observed that the upper bound (\ref{upper}) does not hold for the adjacency
matrix ($\alpha=0$) when $\Delta=3$ if $k(G)=3$ or $k(G)=2$.

Let $\alpha\neq0$ and $\Delta= 3$. Let $k(G)=3$ or $k(G)=2$. From Corollary
\ref{cycleBk}, the spectral radius of $A_{\alpha}(C_{r}\{H\})$ is the spectral
radius of the $3 \times3$ matrix
\[
M_{3}=S_{1}=\left[
\begin{array}
[c]{ccc}%
\alpha & \beta\sqrt{2} & 0\\
\beta\sqrt{2} & 3 \alpha & \beta\\
0 & \beta & \alpha+ 2\\
&  &
\end{array}
\right] .
\]
or of the $2 \times2$ matrix
\[
M_{2}=S_{1}=\left[
\begin{array}
[c]{cc}%
\alpha & \beta\\
\beta & \alpha+ 2\\
&
\end{array}
\right] .
\]
We have
\[
M_{3}=\left[
\begin{array}
[c]{ccc}%
\alpha & \beta\sqrt{2} & 0\\
\beta\sqrt{2} & 3 \alpha & \beta\\
0 & \beta & \alpha+ 2\\
&  &
\end{array}
\right] =\left[
\begin{array}
[c]{ccc}%
3 \alpha & 0 & 0\\
0 & 3 \alpha & 0\\
0 & 0 & 3 \alpha\\
&  &
\end{array}
\right] +\left[
\begin{array}
[c]{ccc}%
-2 \alpha & \beta\sqrt{2} & 0\\
\beta\sqrt{2} & 0 & \beta\\
0 & \beta & 2 \beta\\
&  &
\end{array}
\right] .
\]
Then
\[
\rho(M_{3}) = 3 \alpha+ \beta\rho(Z(\gamma))
\]
where
\[
Z(\gamma)=\left[
\begin{array}
[c]{ccc}%
\gamma & \sqrt{2} & 0\\
\sqrt{2} & 0 & 1\\
0 & 1 & 2\\
&  &
\end{array}
\right]
\]
with $\gamma=-2 \frac{\alpha}{\beta}$. Similarly
\[
M_{2}=\left[
\begin{array}
[c]{cc}%
\alpha & \beta\\
\beta & \alpha+ 2\\
&
\end{array}
\right] =\left[
\begin{array}
[c]{cc}%
3 \alpha & 0\\
0 & 3 \alpha\\
&
\end{array}
\right]  + \beta\left[
\begin{array}
[c]{cc}%
\delta & 1\\
1 & 2\\
&
\end{array}
\right]
\]
where $\delta= -2 \frac{\alpha}{\beta}$. Then
\[
\rho(M_{2}) = 3 \alpha+ \beta\rho(W(\delta))
\]
where
\[
W(\delta) = \left[
\begin{array}
[c]{cc}%
\delta & 1\\
1 & 2\\
&
\end{array}
\right] .
\]

We recall a simplified version of the Weyl's inequalities for eigenvalues of
Hermitian matrices (see, e.g. \cite{HoJo85}, p. 181).

\begin{lemma}
Let $A$ and $B$ be Hermitian matrices of order $n \times n$. Let $C=A+B$. Let
\[
\alpha_{1} \geq\alpha_{2} \geq\ldots\alpha_{n},
\]
\[
\beta_{1} \geq\beta_{2} \geq\ldots\beta_{n}%
\]
and
\[
\gamma_{1} \geq\gamma_{2} \geq\ldots\gamma_{n}%
\]
be the eigenvalues of $A$, $B$ and $C$, respectively. Then, for $j=1,2,\ldots
,n$,
\begin{equation}
\label{Wes}\alpha_{j} + \beta_{n} \leq\gamma_{j} \leq\alpha_{j} + \beta_{1}.
\end{equation}

\end{lemma}

In either of these inequalities equality holds if and only if there exists a
nonzero $n$ -vector that is an eigenvector to each of the three eigenvalues
involved. The conditions for equality in Weyl's inequalities were first
established by So in \cite{So94}.

We have
\[
Z(x)=\left[
\begin{array}
[c]{ccc}%
x-y & 0 & 0\\
0 & 0 & 0\\
0 & 0 & 0\\
&  &
\end{array}
\right]  +Z(y).
\]
We claim that $\rho(Z(\gamma))$ is a strictly increasing function of $\gamma$.
In fact, for $x<y$, Weyl's inequalities imply that
\[
\rho(Z(x))<0+\rho(Z(y))=\rho(Z(y)).
\]
A similar argument shows that $\rho(W(\delta))$ is a strictly increasing
function of $\delta$. Numerical computations show that $\rho(Z(-0.25))<2\sqrt
{2}\cos\frac{\pi}{7}$ and $\rho(Z(-0.2))>2\sqrt{2}\cos\frac{\pi}{7}$ ; and,
$\rho(W(-1.2))<2\sqrt{2}\cos\frac{\pi}{5}$ and $\rho(W(-1.1))>2\sqrt{2}%
\cos\frac{\pi}{5}$. Since $\rho(Z(\gamma))$ and $\rho(W(\delta)$ are
continuous functions, there exists $\gamma_{0}\in(-0.25,-0.2)$ such that
$\rho(Z(\gamma_{0}))=2\sqrt{2}\cos\frac{\pi}{7}$ and there exists $\delta
_{0}\in(-1.2,-1.1)$ such that $\rho(W(\delta_{0}))=2\sqrt{2}\cos\frac{\pi}{5}$.

\begin{theorem}
Let $G$ be a unicyclic graph. Let $\Delta=3$. If $\alpha> \frac{-\gamma_{0}%
}{2-\gamma_{0}}$ whenever $k(G)=3$ or if $\alpha> \frac{-\delta_{0}}%
{2-\delta_{0}}$ whenever $k(G)=2$, then the upper bound (\ref{upper}) holds.
\end{theorem}

\begin{proof}
Let $k(G)=3$. There exists $\gamma_{0}$ such that $\rho(Z(\gamma_{0}))=2
\sqrt{2} \cos\frac{\pi}{7}$. Moreover, $\rho(Z(\gamma))$ is a strictly
increasing function. Hence
\[
\rho(Z(\gamma)) < 2 \sqrt{2} \cos\frac{\pi}{7}%
\]
for $\gamma< \gamma_{0}$. We recall that $\gamma= \frac{-2 \alpha}{1-\alpha}$.
Imposing the inequality
\[
\frac{-2 \alpha}{1-\alpha} < \gamma_{0}%
\]
we obtain $\alpha> \frac{-\gamma_{0}}{2-\gamma_{0}}$. Hence for such values of
$\alpha$, we have
\[
\rho(A_{\alpha}(C_{r}\{H\}))=\rho(M_{3})=3 \alpha+\beta\rho(Z(\gamma)) < 3
\alpha+ 2 (1-\alpha) \sqrt{2} \cos\frac{\pi}{7}.
\]
Then the upper bound (\ref{upper}) holds whenever $k(G)=3$ and $\Delta=3$. The
proof for the case $k(G)=2$ is similar.
\end{proof}

\bigskip\textbf{Acknowledgements. } The author is grateful to the
Mathematical Modeling (CMM), Universidad de Chile, Chile, in which this research was conducted. 

\bigskip

E-mail:

\ \ \ \ Oscar Rojo - \textit{orojo@ucn.cl}

\end{document}